\documentclass[11pt,a4paper]{article}
\usepackage[T1]{fontenc}
\usepackage{lmodern}
\usepackage[margin=1in]{geometry}
\usepackage{amsmath,amssymb,amsthm}
\usepackage{microtype}
\usepackage{tikz}
\usetikzlibrary{arrows.meta,positioning}
\usepackage[colorlinks=true,linkcolor=blue,citecolor=blue,urlcolor=blue]{hyperref}
\hypersetup{pdftitle={Sharp planar Turan bounds for quasi-double stars},pdfauthor={Zehui Shao, Enqiang Zhu, and Shaohui Wang},pdfkeywords={planar Turan number, quasi-double star, planar triangulation, extremal graph theory}}
\newtheorem{theorem}{Theorem}[section]
\newtheorem{lemma}[theorem]{Lemma}
\newtheorem{proposition}[theorem]{Proposition}
\newtheorem{corollary}[theorem]{Corollary}
\theoremstyle{remark}
\newtheorem{remark}[theorem]{Remark}
\newcommand{\ex}{\operatorname{ex}}
\newcommand{\PP}{\mathcal P}
\newcommand{\cdeg}{\operatorname{c}}
\newcommand{\dist}{\operatorname{dist}}
\title{Sharp planar Tur\'an bounds  for quasi-double stars}
\author{Zehui Shao$^{1}$ \quad Enqiang Zhu$^{1}$ \quad Shaohui Wang$^{2}$\\[0.3em]
\small $^{1}$Institute of Computing Science and Technology\\
\small Guangzhou University, Guangzhou 510006, China\\
\small $^{2}$Department of Mathematics, Louisiana Christian University\\
\small Pineville, LA 71359, USA\\
\small E-mail: \href{mailto:shaohui.wang@lcuniversity.edu}{\texttt{shaohui.wang@lcuniversity.edu}} (Shaohui Wang)}
\date{}
\begin{document}
\maketitle
\vspace{-1.5em}
\begin{abstract}
We study $W$-free planar graphs for $W\in\{W_{2,4},W_{2,5},W_{3,4}\}$, where the quasi-double star $W_{h,k}$ is obtained from a three-vertex path by attaching $h$ leaves to one endpoint and $k$ leaves to the other. We prove that every $W_{2,4}$-free planar graph on $n$ vertices has at most $9n/4$ edges, and the bound is attained whenever $8\mid n$. This determines the planar Tur\'an density of $W_{2,4}$ as $9/4$. We also establish the sharp upper bound $5n/2$ for $W_{2,5}$. Combined with known constructions of planar graphs of maximum degree five, it yields $\ex_{\PP}(n,W_{2,5})=\lfloor5n/2\rfloor$ for every $n\ge15$. These results close the two corresponding coefficient gaps in the bounds of Liu et~al. Our proofs use structural restrictions on high-degree vertices, local deletions, and degree deficits in neighborhoods of radius two. For $W_{3,4}$, we characterize the planar graphs with a dominating vertex that avoid this tree and determine their exact extremal number, $\lfloor(5n-7)/2\rfloor$, for every $n\ge10$. Finally, $W$-free planar triangulations have at most eight, twelve, and eleven vertices for $W=W_{2,4},W_{2,5},W_{3,4}$, respectively; the first two bounds are sharp.
\end{abstract}

\section{Introduction}

All graphs are finite and simple. A graph is $H$-free if it contains no subgraph isomorphic to $H$, where containment is not required to be induced. The planar Tur\'an number $\ex_{\PP}(n,H)$ is the maximum number of edges in an $H$-free planar graph on $n$ vertices. Density conditions forcing a prescribed subgraph form a central theme of extremal graph theory; the classical path theorem of Erd\H{o}s and Gallai~\cite{ErdosGallai1959} is a foundational example for trees. The planar problem asks how a prescribed forbidden configuration reduces the edge density available under a planar embedding. Euler's inequality already imposes the universal bound $3n-6$ for $n\ge3$; the central issue is therefore to identify the optimal linear coefficient and, when possible, the exact extremal function and its attaining graphs. This makes planar Tur\'an problems a setting in which local subgraph restrictions must be reconciled with global constraints on the embedding.

Dowden~\cite{Dowden2016} developed this direction for forbidden $4$- and $5$-cycles, and Lan et~al.~\cite{LanShiSong2019} studied a wider range of forbidden graphs, including stars, wheels, and fans. Within the class of trees, Lan and Shi~\cite{LanShi2019} obtained exact results for short paths, while Ghosh et~al.~\cite{GhoshEtAl2022} studied planar Tur\'an bounds for double stars. Some interesting works are in~\cite{b1,b2,b3,b4,b5}. More recently, Xu and Shao~\cite{XuShao2025} determined the sharp coefficient $31/14$ for the double star $S_{2,4}$. Here $S_{h,k}$ denotes the tree obtained from an edge by attaching $h$ leaves to one endpoint and $k$ leaves to the other. These results motivate the study of how the length of the path between branching vertices and the distribution of leaves affect the extremal density.
For trees, degrees give a natural first obstruction: a planar graph of maximum degree less than $\Delta(T)$ is automatically $T$-free. Nevertheless, allowing a few vertices of larger degree need not force a copy of $T$, because the required leaf sets may overlap. Determining the extremal density thus requires more than a maximum-degree estimate; one must understand how the neighborhoods of high-degree vertices can intersect in a planar graph.

For positive integers $h,k$, let $W_{h,k}$ be the tree obtained from a path $xyz$ by adding $h$ leaves adjacent to $x$ and $k$ leaves adjacent to $z$ (Figure~\ref{fig:quasi-double-star}). Thus $W_{h,k}$ has $h+k+3$ vertices; we call $xyz$ its \emph{spine}. Equivalently, $W_{h,k}$ is obtained by subdividing the central edge of a double star once. This subdivision moves the two branching vertices to distance two. Consequently, the existence of a copy depends on the available disjoint leaf sets at two vertices with a common neighbor. The relevant neighborhood comparison is thus between vertices joined by a path of length two. Quasi-double stars therefore provide a natural family for studying the role of common neighborhoods in planar extremal problems.

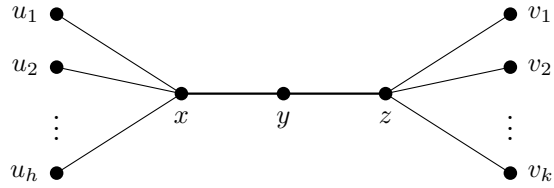
\begin{figure}[htbp]
\centering
\begin{tikzpicture}[x=1cm,y=1cm,
  vertex/.style={circle,fill=black,inner sep=1.8pt},
  every label/.style={font=\small}]
  \coordinate (x) at (-1.35,0);
  \coordinate (y) at (0,0);
  \coordinate (z) at (1.35,0);
  \coordinate (u1) at (-3,1.05);
  \coordinate (u2) at (-3,0.35);
  \coordinate (uh) at (-3,-1.05);
  \coordinate (v1) at (3,1.05);
  \coordinate (v2) at (3,0.35);
  \coordinate (vk) at (3,-1.05);
  \draw[line width=.85pt] (x)--(y)--(z);
  \draw (x)--(u1) (x)--(u2) (x)--(uh)
        (z)--(v1) (z)--(v2) (z)--(vk);
  \node[vertex,label=below:$x$] at (x) {};
  \node[vertex,label=below:$y$] at (y) {};
  \node[vertex,label=below:$z$] at (z) {};
  \node[vertex,label=left:$u_1$] at (u1) {};
  \node[vertex,label=left:$u_2$] at (u2) {};
  \node[vertex,label=left:$u_h$] at (uh) {};
  \node[vertex,label=right:$v_1$] at (v1) {};
  \node[vertex,label=right:$v_2$] at (v2) {};
  \node[vertex,label=right:$v_k$] at (vk) {};
  \node at (-3,-.35) {$\vdots$};
  \node at (3,-.35) {$\vdots$};
\end{tikzpicture}
\caption{The quasi-double star $W_{h,k}$. The spine is $xyz$, with leaves $u_1,\ldots,u_h$ at $x$ and $v_1,\ldots,v_k$ at $z$.}
\label{fig:quasi-double-star}
\end{figure}

Liu et~al.~\cite{LiuXieZhao2026} studied these trees for $1\le h\le2\le k\le5$. For $W_{2,4}$, their results leave a gap between the lower coefficient $9/4$ and the upper coefficient $5/2$; for $W_{2,5}$, the corresponding coefficients are $5/2$ and $17/6$. Their lower bounds arise from disjoint unions of small triangulations. These gaps motivate the following question: can vertices of large degree be arranged to improve those constructions, or does planarity force enough vertices of smaller degree to make the lower coefficients optimal? We answer this question for both trees by proving matching upper bounds. For $W_{2,5}$, the bounded-degree constructions of Lan et~al.~\cite{LanShiSong2019} further turn the sharp coefficient into an exact formula for every $n\ge15$.

We also consider $W_{3,4}$, which has the same order as $W_{2,5}$ but a different distribution of leaves between its branching vertices. As a first step toward its unrestricted extremal problem, we study planar graphs with a dominating vertex. Deleting that vertex leaves an outerplanar graph, so this subclass allows the effect of a single high-degree vertex to be described exactly.

Our first main result determines the planar Tur\'an density of $W_{2,4}$ and gives exact values at every order divisible by eight.

\begin{theorem}\label{thm:w24}
For every positive integer $n$,
\[
\ex_{\PP}(n,W_{2,4})\le\left\lfloor\frac{9n}{4}\right\rfloor.
\]
For every positive integer $t$, equality holds at $n=8t$. Moreover,
\[
\lim_{n\to\infty}\frac{\ex_{\PP}(n,W_{2,4})}{n}=\frac94.
\]
\end{theorem}

The main step is a structural argument for graphs whose degrees lie between three and five. The vertices of degree five induce only a short list of possible connected graphs. We show that distinct nontrivial components on this list have disjoint sets of lower-degree neighbors. A local weight assignment then proves the edge bound.

Our second result determines the planar Tur\'an number of $W_{2,5}$ for all sufficiently large orders.

\begin{theorem}\label{thm:w25}
For every positive integer $n$,
\[
\ex_{\PP}(n,W_{2,5})\le\left\lfloor\frac{5n}{2}\right\rfloor.
\]
Equality holds for every $n\ge15$. In particular, the planar Tur\'an density of $W_{2,5}$ is $5/2$.
\end{theorem}

The upper bound improves the previous coefficient $17/6$ to $5/2$. Its sharpness follows from the star theorem of Lan et~al.~\cite[Theorem~6]{LanShiSong2019}, which provides planar graphs of maximum degree at most five with $\lfloor5n/2\rfloor$ edges for every $n\ge15$. Since $W_{2,5}$ has a vertex of degree six, these graphs are $W_{2,5}$-free.

The proof separates two structural situations. When the vertices of degree six are independent, pairs with a common neighbor admit local deletions at an edge cost of at most $5/2$ per deleted vertex. After excluding such pairs, the neighborhoods of radius two around degree-six vertices have bounded overlap and contain enough degree deficit to prove the bound. When two degree-six vertices are adjacent, we choose such an edge with the fewest common neighbors. A local deletion argument, using this choice in the two densest cases, completes the proof. Thus both lower asymptotic coefficients noted above are optimal.

For $W_{3,4}$, we obtain an exact result in a natural subclass. A vertex is \emph{dominating} if it is adjacent to every other vertex. Write $\ex^{\mathrm{dom}}_{\PP}(n,H)$ for the maximum number of edges in an $H$-free planar graph on $n$ vertices having a dominating vertex.

\begin{theorem}\label{thm:w34}
Let $G$ be a planar graph on $n\ge10$ vertices with a dominating vertex $v$. Then $G$ is $W_{3,4}$-free if and only if $\Delta(G-v)\le3$. Furthermore,
\[
\ex^{\mathrm{dom}}_{\PP}(n,W_{3,4})
=\left\lfloor\frac{5n-7}{2}\right\rfloor
\qquad(n\ge10).
\]
\end{theorem}

In particular, Theorem~\ref{thm:w34} supplies connected $W_{3,4}$-free planar graphs with edge-to-vertex ratio tending to $5/2$. It does not determine the unrestricted planar Tur\'an density of $W_{3,4}$.

Our arguments also bound the orders of triangulations avoiding these trees.

\begin{theorem}\label{thm:triangles}
A $W_{2,4}$-free plane triangulation has at most eight vertices, a $W_{2,5}$-free plane triangulation has at most twelve vertices, and a $W_{3,4}$-free plane triangulation has at most eleven vertices. The first two bounds are attained.
\end{theorem}

The eight-vertex bound is used as a preliminary step in the proof of Theorem~\ref{thm:w24}; it is proved independently of that theorem. The twelve-vertex bound is an immediate consequence of Theorem~\ref{thm:w25}. The eleven-vertex bound follows from restrictions on vertices at distance two from a vertex of large degree.

The main contribution for $W_{2,4}$ and $W_{2,5}$ lies in the upper-bound arguments: the matching lower-bound constructions are already available in the literature. The new structural restrictions explain why allowing vertices of high degree cannot improve their asymptotic coefficients. Figure~\ref{fig:sharp-triangulations} shows explicit triangulations attaining these two bounds; the dominating-vertex construction for $W_{3,4}$ is illustrated in Section~\ref{sec:w34}.

\begin{figure}[htbp]
\centering
\begin{minipage}[b]{.47\textwidth}
\centering
\begin{minipage}[c][4.15cm][c]{\linewidth}
\centering
\begin{tikzpicture}[x=.83cm,y=.83cm,
  every node/.style={font=\small},
  vertex/.style={circle,fill=black,inner sep=1.55pt}]
  \foreach \i in {0,...,5}
    \coordinate (r\i) at ({\i-2.5},0);
  \coordinate (u) at (0,1.45);
  \coordinate (w) at (0,-1.45);
  \foreach \i [evaluate=\i as \j using int(\i+1)] in {0,...,4}
    \draw[line width=.45pt] (r\i)--(r\j);
  \draw[line width=.45pt]
    (r0) .. controls (-2.65,3.15) and (2.65,3.15) .. (r5);
  \foreach \i in {0,...,5} {
    \draw[line width=.45pt] (u)--(r\i)--(w);
    \node[vertex] at (r\i) {};
  }
  \node[vertex] at (u) {};
  \node[vertex] at (w) {};
\end{tikzpicture}
\end{minipage}
\par\smallskip
{\small (a) Eight vertices and eighteen edges.}
\end{minipage}\hfill
\begin{minipage}[b]{.49\textwidth}
\centering
\begin{minipage}[c][4.15cm][c]{\linewidth}
\centering
\begin{tikzpicture}[x=.73cm,y=.73cm,
  every node/.style={font=\small},
  vertex/.style={circle,fill=black,inner sep=1.55pt}]
  \coordinate (P) at (0,3.4641);
  \coordinate (A0) at (-3,-1.73205);
  \coordinate (A1) at (3,-1.73205);
  \coordinate (A2) at (1.25574,.725);
  \coordinate (A3) at (0,1.45);
  \coordinate (A4) at (-1.25574,.725);
  \coordinate (B0) at (0,-1.45);
  \coordinate (B1) at (1.25574,-.725);
  \coordinate (B2) at (.56292,.325);
  \coordinate (B3) at (-.56292,.325);
  \coordinate (B4) at (-1.25574,-.725);
  \coordinate (Q) at (0,-.65);
  \foreach \i in {0,...,4} {
    \pgfmathtruncatemacro{\j}{mod(\i+1,5)}
    \pgfmathtruncatemacro{\k}{mod(\i+4,5)}
    \draw[line width=.45pt]
      (P)--(A\i) (Q)--(B\i)
      (A\i)--(A\j) (B\i)--(B\j)
      (A\i)--(B\i) (A\i)--(B\k);
  }
  \foreach \i in {0,...,4} {
    \node[vertex] at (A\i) {};
    \node[vertex] at (B\i) {};
  }
  \node[vertex] at (P) {};
  \node[vertex] at (Q) {};
\end{tikzpicture}
\end{minipage}
\par\smallskip
{\small (b) Twelve vertices and thirty edges.}
\end{minipage}
\caption{Examples attaining the sharp bounds. In (a), the hexagonal bipyramid is $W_{2,4}$-free because $W_{2,4}$ has nine vertices; $t$ disjoint copies attain Theorem~\ref{thm:w24} at $n=8t$. In (b), the icosahedral graph is $W_{2,5}$-free because every vertex has degree five; $t$ disjoint copies attain Theorem~\ref{thm:w25} at $n=12t$. The constructions for all $n\ge15$ are supplied by Lan et~al.~\cite{LanShiSong2019}.}
\label{fig:sharp-triangulations}
\end{figure}

The paper is organized as follows. Section~\ref{sec:preliminaries} establishes a criterion for embedding a quasi-double star with a specified spine. Sections~\ref{sec:w24} and~\ref{sec:w25} prove the sharp bounds for $W_{2,4}$ and $W_{2,5}$, respectively. Section~\ref{sec:w34} treats the dominating-vertex subclass and triangulations for $W_{3,4}$. We conclude with questions concerning the remaining exact values and the structure of extremal graphs.

\section{Preliminaries}\label{sec:preliminaries}

For a graph $G$, let $v(G)$ and $e(G)$ denote its numbers of vertices and edges. We write $d(x)$, $N(x)$ and $N[x]=N(x)\cup\{x\}$ for the degree, neighborhood and closed neighborhood of a vertex $x$, omitting the graph from the notation when it is clear. Put $N_2(x)=\{z:\dist_G(x,z)=2\}$ and $\cdeg(x,z)=|N(x)\cap N(z)|$. For disjoint vertex sets $A,B$, write $e(A,B)$ for the number of edges between them; when used, $e(A)$ means $e(G[A])$. We use $K_t$, $P_t$ and $C_t$ for the complete graph, path and cycle on $t$ vertices. The octahedral graph is $K_6$ with a perfect matching removed.

The next elementary criterion keeps track of the additional exclusion required when the endpoints of a spine are adjacent in the ambient graph.

\begin{lemma}\label{lem:spine}
Let $xyz$ be a path in a graph $G$, and put $\varepsilon=1$ if $xz\in E(G)$ and $\varepsilon=0$ otherwise. There is a copy of $W_{h,k}$ with spine $xyz$, having $h$ leaves at $x$ and $k$ leaves at $z$, if and only if
\begin{align}
d(x)&\ge h+1+\varepsilon,\qquad d(z)\ge k+1+\varepsilon,\label{eq:spine-degrees}\\
d(x)+d(z)-1-2\varepsilon-\cdeg(x,z)&\ge h+k.\label{eq:spine-union}
\end{align}
\end{lemma}

\begin{proof}
The available leaf sets are $A=N(x)\setminus\{y,z\}$ and $B=N(z)\setminus\{x,y\}$. Their sizes are $d(x)-1-\varepsilon$ and $d(z)-1-\varepsilon$, and their intersection has size $\cdeg(x,z)-1$. The necessary conditions for choosing disjoint subsets of sizes $h$ and $k$ are therefore exactly~\eqref{eq:spine-degrees}--\eqref{eq:spine-union}. They are sufficient as well: choose the $h$ elements of $A$ using as many elements of $A\setminus B$ as possible. The inequalities $|B|\ge k$ and $|A\cup B|\ge h+k$ ensure that at least $k$ elements of $B$ remain.
\end{proof}

We use the planar inequalities $e(G)\le3v(G)-6$ and, for bipartite planar graphs, $e(G)\le2v(G)-4$, whenever the graph has at least three vertices. A plane triangulation means a connected plane graph in which every face, including the outer face, is bounded by a triangle. A simple plane triangulation on at least four vertices has minimum degree at least three, and the neighbors of each vertex appear on a cycle in their cyclic order.

We will also use two elementary embedding observations. First, if $H$ is a 2-connected plane subgraph, each component of $G-V(H)$ lies in a single face of $H$, and all its neighbors in $H$ belong to that face boundary. Second, $G[N(x)]$ is outerplanar: in the plane subgraph induced by $N[x]$, deleting $x$ exposes all the vertices of $N(x)$ to one face.

\begin{lemma}\label{lem:adjacent-three}
A simple plane triangulation on more than four vertices has no adjacent vertices of degree three.
\end{lemma}

\begin{proof}
Suppose that $u,v$ are adjacent and both have degree three. The two facial third vertices on $uv$ are their two remaining neighbors, say $a,b$. The neighborhood of $u$ is a triangle, so $ab$ is an edge. The vertices $u,v,a,b$ induce $K_4$, and their four triangles are facial: two are incident with $uv$, and the other two follow from the degree-three rotations at $u$ and $v$. They cover the sphere, leaving no place for another vertex.
\end{proof}

\section{The planar Tur\'an density of \texorpdfstring{$W_{2,4}$}{W(2,4)}}
\label{sec:w24}

We first record two consequences of Lemma~\ref{lem:spine} that will be
used throughout this section. If $x,z$ are nonadjacent vertices with a
common neighbor, $d(x)\geq5$, and $d(z)\geq3$, then $W_{2,4}$-freeness
implies
\begin{equation}\label{eq:w24nonadj}
 d(x)+d(z)-\cdeg(x,z)\leq6.
\end{equation}
If instead $xz$ is an edge with a common neighbor, $d(x)\geq6$, and
$d(z)\geq4$, then it implies
\begin{equation}\label{eq:w24adj}
 d(x)+d(z)-\cdeg(x,z)\leq8.
\end{equation}
In particular, if $d(x)=5$ and $z$ is a nonadjacent vertex of degree
at least three at distance two from $x$, then
$|N(z)\setminus N(x)|\leq1$.

\subsection{A threshold for triangulations}

\begin{lemma}\label{lem:w24tri}
Every simple plane triangulation on at least nine vertices contains
$W_{2,4}$. The bound nine is best possible.
\end{lemma}

\begin{proof}
Suppose that $G$ is a $W_{2,4}$-free plane triangulation with $n\geq9$.
We have $\delta(G)\geq3$ and $e(G)=3n-6$.

First, suppose that $d(x)\geq7$. A vertex $z$ at distance two from $x$
would satisfy
\[
 d(x)+d(z)-\cdeg(x,z)\geq d(x)\geq7,
\]
contrary to \eqref{eq:w24nonadj}. Thus $x$ is universal, by connectedness,
and $d(x)=n-1\geq8$. If a neighbor $u$ of $x$ has degree at least four,
then $x,u$ have a common neighbor and
$\cdeg(x,u)=d(u)-1$. Consequently
$d(x)+d(u)-\cdeg(x,u)=d(x)+1\geq9$, contrary to
\eqref{eq:w24adj}. All other vertices therefore have degree three,
giving $e(G)=2(n-1)<3n-6$. Hence $\Delta(G)\leq6$.

Suppose next that $d(x)=6$, and put $N=N(x)$ and
$S=V(G)\setminus N[x]$. For a vertex $z$ at distance two from $x$, we have
$d(z)=\cdeg(x,z)$ by \eqref{eq:w24nonadj}, so $N(z)\subseteq N$.
There is therefore no vertex at distance three from $x$. It follows
that $S$ is independent and that every vertex of $S$ has all its
neighbors in $N$.

For $u\in N$, write $s(u)=|N(u)\cap S|$. If $d(u)\geq4$, the edge $xu$
has a common neighbor. By \eqref{eq:w24adj}, we have
\[
 1+s(u)=d(u)-\cdeg(x,u)\leq2.
\]
If $d(u)=3$, the two facial third vertices on $xu$ are its two neighbors
other than $x$, and both belong to $N$; hence $s(u)=0$. Thus $s(u)\leq1$
for every $u\in N$. Since every vertex in $S$ has degree at least three,
we obtain $3|S|\leq6$. As $n\geq9$, equality holds: $n=9$,
$S=\{z_1,z_2\}$, and the sets $U=N(z_1)$ and $V=N(z_2)$ partition $N$
into triples. Both triples induce triangles, being neighborhoods of
degree-three vertices in a triangulation.

The graph has $21$ edges, of which six join $x$ to $N$ and six join $S$
to $N$. Thus $G[N]$ has nine edges, including three edges between $U$
and $V$. Choose $v\in V$ with $k\geq1$ neighbors in $U$. Then
$d(v)=4+k\geq5$ and $\cdeg(v,z_1)=k$. The nonadjacent vertices $v,z_1$
have a common neighbor and satisfy
$d(v)+d(z_1)-\cdeg(v,z_1)=7$, a contradiction to
\eqref{eq:w24nonadj}. This proves $\Delta(G)\leq5$.

The average degree is $6-12/n>4$, so now choose $x$ of degree five.
Put $N=N(x)$, $S=V(G)\setminus N[x]$, and $r=|S|=n-6\geq3$.
Every vertex $z$ at distance two from $x$ satisfies
$|N(z)\setminus N|\leq1$ by \eqref{eq:w24nonadj}. If such a vertex had
a neighbor $w$ at distance three from $x$, a facial third vertex on $zw$
would be a neighbor of $z$ other than $w$, and hence would belong to
$N$. This would place $w$ at distance two, a contradiction. Therefore
all vertices of $S$ are at distance two, and $\Delta(G[S])\leq1$.

The six vertices of $N[x]$ have total degree at most $30$. Hence
\[
 \sum_{z\in S}d(z)\geq6n-42=6r-6\geq4r.
\]
The bipartite planar subgraph consisting of all edges between $N$ and
$S\cup\{x\}$ has at most $2(r+6)-4$ edges. Consequently
\[
 5+4r-2\lfloor r/2\rfloor
 \leq 5+\sum_{z\in S}d(z)-2e(G[S])\leq2r+8.
\]
It follows that $2\lceil r/2\rceil\leq3$, impossible for $r\geq3$.
Finally, any eight-vertex plane triangulation is $W_{2,4}$-free because
$W_{2,4}$ has nine vertices. Such triangulations exist, proving sharpness.
\end{proof}

\subsection{Components with a vertex of degree at least six}

\begin{lemma}\label{lem:w24large}
Let $G$ be a connected $W_{2,4}$-free simple planar graph with
$\delta(G)\geq3$ and $\Delta(G)\geq6$. Then $e(G)\leq9v(G)/4$.
\end{lemma}

\begin{proof}
If $d(x)\geq7$, the argument using \eqref{eq:w24nonadj} in the proof of
Lemma~\ref{lem:w24tri} again shows that $x$ is universal. If $d(x)=7$,
then $v(G)=8$ and, by the planar bound, $e(G)\leq18=9v(G)/4$.
If $d(x)\geq8$, \eqref{eq:w24adj} forces every other vertex to have
degree three, so $e(G)=2(v(G)-1)<9v(G)/4$.

We may therefore assume $\Delta(G)=6$. Choose $x$ of degree six, and put
$N=N(x)$, $S=V(G)\setminus N[x]$, $r=|S|$, and $q=e(N,S)$.
As before, \eqref{eq:w24nonadj} implies that $S$ is independent and
$N(z)\subseteq N$ for every $z\in S$. Applying the bipartite planar bound to the subgraph
between $N$ and $S\cup\{x\}$, we obtain
\begin{equation}\label{eq:w24radius}
 3r\leq q\leq2(r+7)-4-6=2r+4.
\end{equation}
In particular, $r\leq4$. If $r\leq1$, the ordinary planar bound on
$r+7$ vertices proves the assertion. If $r=2$,
Lemma~\ref{lem:w24tri} rules out $21$ edges on nine vertices, so
$e(G)\leq20<9v(G)/4$.

It remains to consider $r\in\{3,4\}$. Write $s(u)=|N(u)\cap S|$ for
$u\in N$, and let $L=\{u\in N:s(u)\leq1\}$. Every edge in $G[N]$ has
both endpoints in $L$: if $u$ has a neighbor in $N$ and $s(u)\geq2$,
then $x,u$ have a common neighbor, $d(u)\geq4$, and
\[
 d(x)+d(u)-\cdeg(x,u)=7+s(u)\geq9,
\]
contrary to \eqref{eq:w24adj}.
For $r=3$, we have $9\leq q\leq10$. If $|L|\geq5$, then
$q=\sum_{u\in N}s(u)\leq5+3=8$, a contradiction. Hence $|L|\leq4$
and $e(G[N])\leq6$, yielding $e(G)=6+q+e(G[N])\leq22<9\cdot10/4$.
For $r=4$, we have $q=12$ by \eqref{eq:w24radius}. Again, $|L|\geq5$ would
imply $q\leq5+4=9$. Thus $e(G[N])\leq6$ and
$e(G)\leq6+12+6=24<9\cdot11/4$.
\end{proof}

\subsection{The components induced by vertices of degree five}

\begin{lemma}\label{lem:w24bounded}
Every $W_{2,4}$-free simple planar graph $G$ with
$3\leq\delta(G)\leq\Delta(G)\leq5$ satisfies $e(G)\leq9v(G)/4$.
\end{lemma}

\begin{proof}
Call a vertex \emph{high} if it has degree five, and \emph{low}
otherwise. Let $F$ be the subgraph induced by the high vertices.
Assign to each low vertex $v$ the weight $w(v)=9-2d(v)$, which is three
for degree three and one for degree four. It suffices to show that the
total low weight is at least $v(F)$, since
\begin{equation}\label{eq:w24weight}
 9v(G)-4e(G)=\sum_{v\notin V(F)}w(v)-v(F).
\end{equation}

\smallskip\noindent
\textit{Component classification.}
If high vertices $x,z$ are at distance two in $G$, then
$\cdeg(x,z)\geq4$ by \eqref{eq:w24nonadj}. If $x$ had two high
vertices $z,z'$ at distance two, their common-neighbor sets with $x$
would be two subsets of the five-set $N(x)$, each of size at least four.
Their intersection would contain at least three vertices, producing a
$K_{3,3}$ with the other part $\{x,z,z'\}$. Thus every high vertex has
at most one high vertex at distance two in $G$, and the same property
holds in $F$.

A connected graph $J$ with at most one distance-two vertex from each
vertex is either $P_4$ or a complete graph with a matching deleted.
Indeed, a geodesic of length four is impossible, because its middle
vertex would have two distance-two vertices. If the diameter is at most
two, every vertex has at most one nonneighbor, giving the latter form.
Otherwise, take a geodesic $a b c d$. An additional neighbor $u$ of $b$
would have to be adjacent to $a$, whose unique distance-two vertex is
$c$, and to $c$, whose unique distance-two vertex is $a$. It would then
have to be adjacent to $d$, whose unique distance-two vertex is $b$.
This puts $a,d$ at distance two, a contradiction. Hence
$N_J(b)=\{a,c\}$, and symmetrically $N_J(c)=\{b,d\}$. Any additional
neighbor of $a$ would have to be adjacent to $b$, again a contradiction;
the same holds for $d$. By connectedness, we conclude that $J=P_4$.

By planarity, the components of $F$ belong to the following eleven types:
\begin{equation}\label{eq:w24types}
 \begin{gathered}
 K_1,\ K_2,\ P_3,\ P_4,\ K_3,\ C_4,\ K_4-e,\ K_4,\\
 K_5-e,\ K_5-2M,\ K_6-3M.
 \end{gathered}
\end{equation}
Here $jM$ denotes a matching of size $j$; $K_6-3M$ is the octahedral
graph. To see that the list is exhaustive, a complete graph of order
$h$ with a matching removed has at least
$\binom h2-\lfloor h/2\rfloor$ edges, exceeding $3h-6$ for $h\geq7$.
At order six the entire perfect matching must be removed, and at order
five at least one edge must be removed. The smaller orders give exactly
the remaining types in \eqref{eq:w24types}.

\smallskip\noindent
\textit{Separation of low neighborhoods.}
For a nontrivial component $H$ of $F$, let $L(H)$ be its set of low
neighbors. We claim that no member of $L(H)$ is adjacent to a high
vertex outside $H$.
Suppose otherwise that high vertices $x\in V(H)$ and $z\notin V(H)$
have a common low neighbor. They are nonadjacent and have at least four
common neighbors, all low; a common high neighbor would join their
components of $F$. Choose a high neighbor $p$ of $x$. Then their common
neighborhood $C$ has exactly four vertices, and
\[
 N(x)=C\cup\{p\},\qquad N(z)=C\cup\{w\}
\]
for some vertex $w$. The vertex $p$ is adjacent to no member of $C$.
Otherwise $p,z$ would be nonadjacent high vertices at distance two, so
at least four members of $N(z)$ would be adjacent to $p$. At least three
would belong to $C$, giving a $K_{3,3}$ with the other part $\{x,z,p\}$.

For each $c\in C$, the nonadjacent vertices $p,c$ have common neighbor
$x$. Thus, by \eqref{eq:w24nonadj}, we have $|N(c)\setminus N(p)|\leq1$.
The vertex $z$ already belongs to this difference. Since $d(c)\geq3$,
the vertex $c$ has a neighbor in $N(p)\setminus\{x\}$.
The set
\[
 T=\{p\}\cup\bigl(N(p)\setminus\{x\}\bigr)
\]
is connected and is disjoint from $\{x,z\}\cup C$: the vertices $p,z$
are in different components of $F$, and $p$ has no neighbor in $C$.
Contracting $T$ to one vertex gives three pairwise disjoint branch sets,
$T$, $\{x\}$, and $\{z\}$, each adjacent to all four vertices of $C$.
Any three vertices of $C$ complete a $K_{3,3}$ minor, a contradiction.
This proves the claim. Consequently the sets $L(H)$ for nontrivial
components are pairwise disjoint, and none of their members is adjacent
to an isolated vertex of $F$.

\smallskip\noindent
\textit{Weights for nontrivial components.}
We next prove, for each nontrivial component $H$,
\begin{equation}\label{eq:w24payment}
 \sum_{v\in L(H)}w(v)\geq v(H).
\end{equation}
Put $h=v(H)$ and $m=e(H,L(H))=5h-2e(H)$. For
$H=K_2,P_3,P_4,K_3$, the values of $m$ are respectively $8,11,14,9$.
Each low vertex has at most $h$ neighbors in $H$, so $|L(H)|$ is at
least $4,4,4,3$, respectively. As all low weights are at least one,
\eqref{eq:w24payment} follows. If $H=C_4$, then $m=12$. Applying the
bipartite planar bound to the graph between $H$ and $L(H)$, we obtain
$12\leq2(4+|L(H)|)-4$, hence $|L(H)|\geq4$.

For the five remaining types we use a fixed plane embedding of $G$.
Every connected component of $G-V(H)$ lies in one face of the embedded
2-connected graph $H$, and all its neighbors in $H$ lie on that face
boundary. In particular, a vertex of $H$ outside that boundary has no
neighbor among the vertices of $G-V(H)$ in the face. This follows
directly from the separation by the facial boundary cycles.

\smallskip\noindent
\emph{Case $H=K_4-e$.}
Let $a,b$ be the vertices of degree three in $H$, and $x,y$ the vertices
of degree two. The face boundaries of $H$ are $abx$, $aby$, and $xayb$.
Here $m=10$, so $|L(H)|\geq3$. Suppose $|L(H)|=3$. Two low vertices
each adjacent to all four vertices of $H$ are impossible: together with
$H$, they would give at least $5+8=13$ edges on six vertices. Therefore
the three numbers of neighbors in $H$ must be $4,3,3$.
Let $v$ have four neighbors in $H$. It lies in the quadrilateral face,
and its four incident edges to $H$ divide that face into triangular
regions, each having only two vertices of $H$ on its boundary. Thus
neither of the other two low vertices, each with three neighbors in
$H$, can lie there. They lie in the original triangular faces and are
both adjacent to $a,b$. Together with $v$, they give $a$ and $b$ at
least three neighbors outside $H$, although each has exactly two.
This contradiction yields $|L(H)|\geq4$.

\smallskip\noindent
\emph{Case $H=K_4$.}
Here $m=8$, and a low vertex has at most three neighbors in $H$, so
$|L(H)|\geq3$. Only equality needs consideration. Some $v\in L(H)$
has three neighbors $a,b,c$ in $H$ and lies in the face $abc$.
The fourth vertex $d$ of $H$ has two low neighbors; these must be the
other two members of $L(H)$, both lying outside that face.
If $v$ had a neighbor $u$ outside $H$, it would lie in the face $abc$,
and hence would not belong to $L(H)$. In particular, $u$ has no neighbor
in $H$. For any $a\in\{a,b,c\}$, its only possible common neighbor with
$u$ is $v$: every such common neighbor must be a low neighbor of $a$
in the face $abc$, and $v$ is the only member of $L(H)$ there.
But $a,u$ are nonadjacent, $d(a)=5$, and $d(u)\geq3$, contradicting
\eqref{eq:w24nonadj}. Thus $d(v)=3$, and the total low weight is at
least $3+1+1>4$.

\smallskip\noindent
\emph{Case $H=K_5-e$.}
View $H$ as a triangular bipyramid with base triangle $abc$ and
nonadjacent apices $x,y$. Each base vertex has degree four in $H$ and
therefore a unique low neighbor. The low neighbor $v$ of $a$ lies in a
triangular face, which we may label $abx$. The other apex $y$ lies
outside this face. By face separation, its only possible common
neighbors with $v$ are $a,b$, and $a$ is a common neighbor. Applying
\eqref{eq:w24nonadj} to $y,v$, we obtain $d(v)\leq3$, so $d(v)=3$.
Here $m=7$, and every low vertex has at most three neighbors in $H$;
thus $|L(H)|\geq3$, with total weight at least $3+1+1=5$.

\smallskip\noindent
\emph{Case $H=K_5-2M$.}
This graph is a wheel with hub $u$ and rim cycle $abcd$. The hub has a
unique low neighbor $v$, lying in a triangular face, say $uab$.
The rim vertex $c$ lies outside this face and can share with $v$ only
the boundary vertices $u,b$. Since $u$ is common, \eqref{eq:w24nonadj}
again forces $d(v)=3$. Here $m=9$, and a low vertex has at most four
neighbors in $H$, so $|L(H)|\geq3$. The total weight is at least five.

\smallskip\noindent
\emph{Case $H=K_6-3M$.}
Every vertex of this octahedral graph has degree four in $H$, so $m=6$.
A low neighbor $v$ lies in a triangular face $abc$; relabel so that
$a$ is a neighbor of $v$. The octahedral vertex $b'$ opposite $b$ lies
outside this face, is adjacent to $a$, and has precisely $a,c$ as
neighbors on the face boundary. Thus $v,b'$ have a common neighbor
and at most two common neighbors. By \eqref{eq:w24nonadj},
we have $d(v)=3$. Every member of $L(H)$ therefore has weight three.
As each has at most three neighbors in $H$ and $m=6$, there are at
least two of them, providing total weight at least six.

This proves \eqref{eq:w24payment} in every nontrivial case.
Finally, each isolated vertex of $F$ has five low neighbors. Assign
$1/5$ of a unit to each of those neighbors on behalf of that high
vertex. A low vertex of degree three is assigned at most $3/5$, and
one of degree four at most $4/5$, in either case no more than its weight.
By the separation claim, none of these low vertices belongs to $L(H)$
for a nontrivial component. Thus these assignments and
\eqref{eq:w24payment} together account for every high vertex without
exceeding the total low weight. Equation~\eqref{eq:w24weight} completes
the proof.
\end{proof}

\begin{proof}[Proof of Theorem~\ref{thm:w24}]
We prove the upper bound by induction on $n$, starting with the empty
graph. If a $W_{2,4}$-free planar graph $G$ has a vertex $v$ of degree at
most two, then
\[
 e(G)\leq e(G-v)+2\leq\frac94(n-1)+2\leq\frac94n.
\]
Otherwise each component has minimum degree at least three.
Lemma~\ref{lem:w24large} applies to components of maximum degree at
least six, and Lemma~\ref{lem:w24bounded} applies to all remaining
components. Summing their edge bounds proves the upper bound.

For each positive integer $t$, take the disjoint union of $t$
eight-vertex triangulations, for instance the graph in Figure~\ref{fig:sharp-triangulations}(a). The resulting graph is $W_{2,4}$-free and has $8t$ vertices
and $18t$ edges. Adding isolated vertices proves
\[
 18\left\lfloor\frac n8\right\rfloor
 \leq\ex_{\PP}(n,W_{2,4})\leq\left\lfloor\frac{9n}{4}\right\rfloor.
\]
In particular, equality with $9n/4$ holds whenever $8\mid n$, and
$\lim_{n\to\infty}\ex_{\PP}(n,W_{2,4})/n=9/4$.
\end{proof}

\section{The planar Tur\'an number of \texorpdfstring{$W_{2,5}$}{W(2,5)}}\label{sec:w25}

We prove the sharp bound $e(G)\le5v(G)/2$ for every $W_{2,5}$-free planar graph. We first record the distance-two restriction that drives the argument. If $d(x)=6$ and $z\in N_2(x)$, then
\begin{equation}\label{eq:w25-distance-two}
|N(z)\setminus N(x)|\le1.
\end{equation}
Indeed, two neighbors of $z$ outside $N(x)$, a common neighbor of $x,z$, and the five remaining neighbors of $x$ would give a copy of $W_{2,5}$.

\begin{lemma}\label{lem:w25-independent}
Let $G$ be a $W_{2,5}$-free planar graph of maximum degree at most six. If its degree-six vertices form an independent set, then $e(G)\le5v(G)/2$.
\end{lemma}

\begin{proof}
We use induction on $v(G)$, with the empty graph as the initial case. Deleting a vertex of degree at most two preserves the hypotheses, so assume that $\delta(G)\ge3$. Let $H$ denote the set of degree-six vertices. Each $x\in H$ has at most one vertex of $H$ at distance two. In fact, if $y,z$ were two such vertices, then, by~\eqref{eq:w25-distance-two}, we would have
\[
|N(x)\cap N(y)|\ge5,\qquad |N(x)\cap N(z)|\ge5.
\]
These two subsets of $N(x)$, each of size at least five, have at least four common elements. Any three of those elements, together with $x,y,z$, give $K_{3,3}$.

Suppose first that $x,y\in H$ have distance two. Put $C=N(x)\cap N(y)$, $t=|C|\in\{5,6\}$, $A=N(x)\setminus C$, $B=N(y)\setminus C$, and
\[
X=\{x,y\}\cup N(x)\cup N(y),\qquad
Z=\{z\notin X:N(z)\cap X\ne\varnothing\}.
\]
Write $r=|Z|$. All vertices of $X\setminus\{x,y\}$ have degree at most five, because $H$ is independent. We show that $D=X\cup Z$ meets at most $5|D|/2$ edges. Induction then applies to $G-D$, since deletion cannot create adjacent degree-six vertices.

If $t=6$, then $A=B=\varnothing$ and $|X|=8$. Every $z\in Z$ has at most two neighbors in $C$, by the $K_{3,3}$ obstruction, and at most one neighbor outside $C$, by~\eqref{eq:w25-distance-two}. Thus $d(z)=3$, with two neighbors in $C$. Writing $m=e(G[C])$ and summing the degrees at $C$, we obtain
\[
12+2m+2r\le30.
\]
Hence $e(G[X])=12+m\le21-r$. We also have $e(G[X])\le18$ by Euler's bound. The number of edges meeting $D$ is at most
\[
e(G[X])+\sum_{z\in Z}d(z)
\le\min\{18+3r,21+2r\}.
\]
The first bound is at most $5(8+r)/2$ when $r\le4$, and the second when $r\ge2$. These ranges cover all possibilities.

If $t=5$, write $A=\{a\}$ and $B=\{b\}$, so $|X|=9$. For $z\in Z$, let $a_z,b_z,c_z,o_z$ count its neighbors in $A,B,C,V(G)\setminus X$, respectively. We have $c_z\le2$. By~\eqref{eq:w25-distance-two},
\[
b_z+o_z\le1\quad\text{if }a_z+c_z>0,
\qquad
a_z+o_z\le1\quad\text{if }b_z+c_z>0.
\]
Together with $\delta(G)\ge3$, these imply exactly three possibilities: $c_z=1,a_z=b_z=1,o_z=0$; $c_z=2,a_z=b_z=0,o_z=1$; or $c_z=2,o_z=0,1\le a_z+b_z\le2$. Thus every vertex of $Z$ has degree three or four. Let $j$ count those of degree four, and let $u$ count those having a neighbor outside $X$. Then
\[
j+u\le r,\qquad
\sum_{z\in Z}d(z)=3r+j,\qquad
e(X,Z)=3r+j-u.
\]
The seven vertices in $L=X\setminus\{x,y\}$ have degree at most five. There are twelve edges between $\{x,y\}$ and $L$. Writing $m=e(G[L])$, we obtain
\[
12+2m+3r+j-u\le35.
\]
For $r\ge2$, the number of edges meeting $D$ is consequently at most
\[
12+m+3r+j
\le\frac{47}{2}+\frac{3r}{2}+\frac{j+u}{2}
\le\frac{47}{2}+2r
\le\frac52(9+r).
\]
For $r=0$, Euler's bound gives at most twenty-one edges. For $r=1$, the same bound on $G[X]$, together with the outside vertex's degree at most four, gives at most twenty-five edges. Both are at most $5|D|/2$.

We may therefore assume that distinct vertices of $H$ have distance at least three. Their open neighborhoods are pairwise disjoint. For $v\notin H$, define its nonnegative deficit $b(v)=5-d(v)$. We claim that, for each $x\in H$,
\begin{equation}\label{eq:w25-deficit}
\sum_{v\in N(x)\cup N_2(x)}b(v)\ge2.
\end{equation}
Fix $x$, put $S=N(x)$, $T=N_2(x)$ and $r=|T|$, and write $D_S=\sum_{v\in S}b(v)$ and $D_T=\sum_{v\in T}b(v)$. Neither $S$ nor $T$ meets $H$. Every vertex of $T$ has at most one neighbor outside $S$, and every neighbor of $S$ belongs to $U=\{x\}\cup S\cup T$. Thus $U$ sends at most $r$ edges to its complement. Applying Euler's bound to $G[U]$, we have
\[
36+5r-D_S-D_T-r
\le2e(G[U])\le6(7+r)-12,
\]
so $D_S+D_T\ge6-2r$. The bipartite planar subgraph between $S$ and $\{x\}\cup T$ has at least $6+4r-D_T$ edges and at most $2(7+r)-4$. Hence $D_T\ge2r-4$. If $r\le2$, the first inequality proves~\eqref{eq:w25-deficit}; if $r\ge3$, the second does.

Every vertex outside $H$ belongs to at most one set $N(x)$ with $x\in H$. It also belongs to at most one set $N_2(x)$. Otherwise, if $v\in N_2(x)\cap N_2(y)$ for distinct $x,y\in H$, then, by~\eqref{eq:w25-distance-two} and the disjointness of $N(x),N(y)$, we obtain
\[
2d(v)-2\le
|N(v)\cap N(x)|+|N(v)\cap N(y)|\le d(v),
\]
contrary to $d(v)\ge3$. Summing~\eqref{eq:w25-deficit} therefore yields
\[
2|H|\le2\sum_{v\notin H}(5-d(v)).
\]
Finally,
\[
2e(G)=5v(G)+|H|-\sum_{v\notin H}(5-d(v))\le5v(G),
\]
as required.
\end{proof}

\begin{proof}[Proof of Theorem~\ref{thm:w25}]
We use strong induction on $n$, including the empty graph. A vertex of degree at most two can be deleted, so assume that $\delta(G)\ge3$. More generally, a nonempty set $D$ meeting at most $5|D|/2$ edges can be deleted and the induction hypothesis applied to $G-D$.

We first dispose of vertices of degree at least seven. Suppose that $d(x)\ge8$. If $z\in N_2(x)$, choose a common neighbor $y$ of $x,z$ and two neighbors of $z$ different from $y$. At least five neighbors of $x$ remain after excluding these three vertices, giving $W_{2,5}$. Hence $N[x]$ is a component. If $d(x)=8$, this component has nine vertices and at most $21<5\cdot9/2$ edges. If $d(x)\ge9$, a vertex $u\in N(x)$ cannot have three neighbors in $N(x)$: one would serve as the middle of a spine from $u$ to $x$, the other two as leaves at $u$, leaving at least $d(x)-4\ge5$ leaves at $x$. Thus $\Delta(G[N(x)])\le2$, and the component has at most $2d(x)<5(d(x)+1)/2$ edges. In either case the component is deletable.

Next suppose that $d(x)=7$. Every $z\in N_2(x)$ has all its neighbors in $N(x)$. Otherwise, choose $a\in N(z)\setminus N(x)$, a common neighbor $y$ of $x,z$, and a further neighbor $b$ of $z$. The leaves $a,b$ at $z$ exclude at most one neighbor of $x$ other than $y$, again giving $W_{2,5}$. Consequently the component of $x$ has vertex set $N[x]\cup N_2(x)$, with $N_2(x)$ independent. If $r=|N_2(x)|$, applying the bipartite planar bound to the graph with parts $N(x)$ and $\{x\}\cup N_2(x)$ yields
\[
7+3r\le2(8+r)-4,
\]
so the component has order $v\le13$. If $v\le12$, we have $e\le3v-6\le5v/2$ by Euler's bound. If $v=13$, the only obstruction to the desired deletion would be a triangulation with thirty-three edges. In such a triangulation, each $u\in N(x)$ has at most one neighbor in $N_2(x)$. Indeed, if $d(u)\ge4$, the adjacent endpoints $x,u$ have a common neighbor, and Lemma~\ref{lem:spine} and freeness force
\[
d(u)-\cdeg(x,u)\le2.
\]
One neighbor counted outside $N(x)$ is $x$ itself, leaving room for at most one vertex of $N_2(x)$. If $d(u)=3$, the two facial third vertices on $xu$ are the other two neighbors of $u$, so $u$ has no neighbor in $N_2(x)$. It follows that $3|N_2(x)|\le7$, contrary to $v=13$ and $|N_2(x)|=5$. Thus the component has at most thirty-two edges and is deletable in this case as well.

We may now assume that all degrees lie between three and six. If the degree-six vertices are independent, Lemma~\ref{lem:w25-independent} proves the result. Otherwise choose an edge $xy$ with $d(x)=d(y)=6$ such that $\cdeg(x,y)$ is as small as possible among all such edges. Set
\begin{align*}
C&=N(x)\cap N(y), &t&=|C|,\\
A&=N(x)\setminus N[y], &B&=N(y)\setminus N[x],\\
s&=|A|=|B|=5-t, &X&=\{x,y\}\cup A\cup B\cup C,
\end{align*}
and let $R=V(G)\setminus X$. Thus $|X|=7+s$. Every $a\in A$ has at most one neighbor in $A\cup R$, since two such neighbors would be the small leaves on the spine $axy$, with $N(y)\setminus\{x\}$ providing the five large leaves. The analogous assertion holds for $B$, so
\begin{equation}\label{eq:w25-private}
2e(A)+e(A,R)\le s,\qquad 2e(B)+e(B,R)\le s.
\end{equation}

Suppose first that $t\le1$. The bipartite planar subgraph consisting of the edges between $A\cup\{y\}$ and $B\cup\{x\}$ has at most $4s$ edges. At most $6t$ edges meet $C$. All remaining edges meeting $X$ lie inside $A$ or $B$, or join these sets to $R$; their number is at most $2s$ by~\eqref{eq:w25-private}. Thus $X$ meets at most
\[
4s+6t+2s=30
\]
edges. For $t=0$, its order is twelve and it is deletable. For $t=1$, its order is eleven. Every vertex outside $X$ adjacent to $X$ lies at distance two from $x$ or $y$, and therefore has at most one neighbor in $R$. If there are at least two such vertices, deleting $X$ together with any two of them removes at most thirty-two edges on thirteen vertices. If there are none, Euler's bound gives at most twenty-seven edges on eleven vertices. Finally, suppose that there is exactly one such vertex $z$, and put $D=X\cup\{z\}$. The graph $G[D]$ has twelve vertices and cannot be a triangulation: the edge $xy$ has only one common neighbor even in $G$, whereas an edge of a simple triangulation on at least four vertices has two distinct facial third vertices. Hence $e(G[D])\le29$. At most one edge joins $D$ to its complement, and deleting $D$ removes at most thirty edges. All these deletions meet the required ratio.

Henceforth $2\le t\le5$ and $0\le s\le3$. Put
\[
Z=\{z\in R:N(z)\cap X\ne\varnothing\}.
\]
For $z\in Z$, write $a_z,b_z,c_z,o_z$ for its numbers of neighbors in $A,B,C,R$, respectively. We have $c_z\le2$, since three common neighbors of $x,y,z$ would give $K_{3,3}$. Moreover, by~\eqref{eq:w25-distance-two}, we have
\begin{equation}\label{eq:w25-outside}
\begin{split}
b_z+o_z&\le1\quad\text{if }a_z+c_z>0,\\
a_z+o_z&\le1\quad\text{if }b_z+c_z>0.
\end{split}
\end{equation}
Together with $d(z)\ge3$, these inequalities lead to the following exhaustive list.
\begin{itemize}
\item If $c_z=0$, exactly one of $a_z,b_z$ is positive; it is at least two, and $o_z\le1$. The degree is three or four, and degree four requires three neighbors in one private set and one neighbor in $R$.
\item If $c_z=1$, then $a_z=b_z=1$ and $o_z=0$, so the degree is three.
\item If $c_z=2$, then either $o_z=1$ and $a_z=b_z=0$, or $o_z=0$, $a_z,b_z\le1$, and $a_z+b_z\ge1$. Degree four occurs exactly when $a_z=b_z=1$.
\end{itemize}
For the first item, positive values of both $a_z,b_z$ would force $a_z=b_z=1,o_z=0$, contrary to the minimum degree. For the second, a positive $o_z$ would force $a_z=b_z=0$ and degree at most two. These observations verify that the list includes every possibility.

Let $p,q,r_2$ count the vertices in the three classes, let $r=|Z|=p+q+r_2$, and let $j_0,j_2$ count the degree-four vertices in the first and third classes. Put $j=j_0+j_2$. Counting private incidences, we obtain
\begin{equation}\label{eq:w25-joint}
2p+2q+j_0+2j_2\le2s.
\end{equation}
Here a first-class vertex uses at least two private incidences and uses three if its degree is four; a second-class vertex uses two; and a degree-four third-class vertex uses two. Since a private set has at most $s\le3$ available incidences, it supports at most one first-class vertex. Hence $p\le2$, while $p=0$ if $s\le1$. For $s\le2$, a private set is too small to support a degree-four first-class vertex, so $j_0=0$.

In fact, $j_0=0$ also when $s=3$. To see this, suppose that a first-class vertex $z$ has degree four, and, after interchanging $A,B$ if necessary, is adjacent to all three vertices of $A$ and to one vertex of $R$. Each vertex of $A$ has used its unique allowed neighbor in $A\cup R$ on $z$. Thus all its remaining neighbors lie in $\{x\}\cup B\cup C$. Consider the plane subgraph $K_{2,3}$ with parts $\{x,z\}$ and $A$. The set $\{y\}\cup B\cup C$ is connected outside this subgraph, since $y$ is adjacent to every vertex of $B\cup C$. It therefore lies in a single face. Each face of this $K_{2,3}$ has just two vertices of $A$ on its boundary. The third vertex of $A$ can consequently have no neighbor in $B\cup C$, leaving it with only $x,z$ as neighbors. This contradicts $\delta(G)\ge3$. Hence $j_0=0$ in every remaining case, and $j=j_2$. It now follows from~\eqref{eq:w25-joint} that $p+q+j\le s$.

Let $J\subseteq Z$ be the set of degree-four vertices, so $|J|=j$. Every vertex of $J$ has exactly four neighbors, all in $X$, and there are no edges within $J$. Applying Euler's bound to $G[X\cup J]$, we have
\begin{equation}\label{eq:w25-jEuler}
e(G[X])+4j\le3(|X|+j)-6,
\qquad e(G[X])+j\le15+3s.
\end{equation}
We also need an improvement independent of $j$ when $s=3$.
In that case,
\begin{equation}\label{eq:w25-X23}
e(G[X])+j\le23.
\end{equation}
Indeed, the graph $G[X\cup J]$ cannot be a triangulation. Otherwise the six neighbors of $x$, which are unchanged by adding $J$, would form a cycle in their cyclic order. The two neighbors of $y$ on this cycle must be the two vertices of $C$, because $y$ has no neighbor in $A$. The complementary segment between those two vertices contains all three vertices of $A$ consecutively. It includes two edges of $G[A]$, contradicting $2e(A)\le3$ from~\eqref{eq:w25-private}. Thus $e(G[X])+4j\le3(10+j)-7$, proving~\eqref{eq:w25-X23}.

All edges from $X$ to its complement end at $Z$. Consequently the number $M$ of edges meeting $X\cup Z$ is
\begin{equation}\label{eq:w25-deletion}
M=e(G[X])+\sum_{z\in Z}d(z)-e(G[Z])
\le e(G[X])+3r+j.
\end{equation}
There is a second useful bound when $s\ge1$. Let $I=e(A)+e(B)$. Counting the private incidences in~\eqref{eq:w25-private} and using $j_0=0$, we obtain
\[
I\le s-p-q-j.
\]
The bipartite subgraph between $A\cup\{y\}$ and $B\cup\{x\}$ has at most $4s$ edges. Inside $X$, at most $6t-q-2r_2$ edges meet $C$, by subtracting its incidences with $Z$ from its degree budget. All remaining edges of $G[X]$ are counted by $I$. Consequently
\begin{align}
e(G[X])&\le4s+I+6t-q-2r_2,\notag\\
M&\le30-s+r+p.
\label{eq:w25-budget}
\end{align}
Here $s\ge1$ ensures that the bipartite graph has at least four vertices.

We now complete the deletion argument, considering the four values of $s$.

Suppose first that $s=0$, so $t=5$, $p=q=j=0$, and $|X|=7$. Write $m=e(G[C])$. Then $e(G[X])=11+m$, so Euler's bound implies $m\le4$. Summing the degrees at $C$, we obtain
\[
10+2m+2r\le30,\qquad m+r\le10.
\]
Thus
\[
M\le\min\{15+3r,21+2r\}.
\]
The first expression is at most $5(7+r)/2$ when $r\le5$, and the second when $r\ge7$. For $r=6$, a failure would require $m=4$, since $m\le3$ implies $M\le32<5\cdot13/2$. If $m=4$, equality holds in the degree sum at $C$, so all five vertices of $C$ have degree six. At least one of them, say $c$, has a neighbor in $Z$. Since $A=\varnothing$ and $d(c)=6$, this implies
\[
\cdeg(x,c)=1+|N(c)\cap C|\le4<5=\cdeg(x,y),
\]
contradicting the choice of $xy$. This completes the case $s=0$.

Next let $s=1$, so $t=4$ and $|X|=8$. Here $p=0$ and $q+j\le1$. For $r\le4$, it follows from~\eqref{eq:w25-deletion} and~\eqref{eq:w25-jEuler} that
\[
M\le18+3r\le\frac52(8+r).
\]
For $r\ge6$, we have, by~\eqref{eq:w25-budget},
\[
M\le29+r\le\frac52(8+r).
\]
It remains to handle $r=5$. Let $h$ count the degree-six vertices in $C$. Each such vertex $c$ has at most one neighbor in $Z$, since otherwise
\[
\cdeg(x,c)\le d(c)-1-|N(c)\cap Z|\le3<\cdeg(x,y),
\]
again contradicting the choice of $xy$. Every other vertex of $C$ has degree at most five and hence at most three neighbors in $Z$. Therefore
\[
10-q=q+2r_2=e(C,Z)\le h+3(4-h)=12-2h.
\]
Since $q\le1$, it follows that $h\le1$. The graph $G[X\setminus C]$ has at most four edges: writing $A=\{a\}$ and $B=\{b\}$, the only possible edges are $xy,xa,yb,ab$. Summing the degrees at $C$, we now obtain
\[
e(G[X])\le4+20+h-q-2r_2=24+h-2r+q.
\]
Consequently
\[
M\le24+h+r+q+j\le31<\frac52\cdot13.
\]

Suppose that $s=2$, so $t=3$ and $|X|=9$. For $r\le3$, equations~\eqref{eq:w25-deletion} and~\eqref{eq:w25-jEuler} yield
\[
M\le21+3r\le\frac52(9+r).
\]
For $r\ge5$, combining~\eqref{eq:w25-budget} with $p\le2$, we have
\[
M\le30+r\le\frac52(9+r).
\]
Consider $r=4$. If $p=0$, we obtain $M\le32<5\cdot13/2$ from~\eqref{eq:w25-budget}. If $p>0$, choose a first-class vertex $z$, and, after interchanging $A,B$ if necessary, suppose that it is adjacent to both vertices of $A$. These two adjacencies exhaust the available incidences from $A$ to $A\cup R$, and therefore $e(A)=0$. The graph $G[X\cup J]$ cannot be a triangulation, because the outside vertex $z$ has the two nonadjacent vertices of $A$ as neighbors, which cannot both belong to the boundary of a triangular face. Hence
\[
e(G[X])+4j\le3(9+j)-7,
\qquad e(G[X])+j\le20.
\]
By~\eqref{eq:w25-deletion}, we therefore have $M\le20+3\cdot4=32$, as required.

Finally let $s=3$, so $t=2$ and $|X|=10$. For $r\le4$, it follows from~\eqref{eq:w25-deletion} and~\eqref{eq:w25-X23} that
\[
M\le23+3r\le\frac52(10+r).
\]
For $r\ge5$, we have, by~\eqref{eq:w25-budget},
\[
M\le27+r+p\le29+r\le\frac52(10+r).
\]
Every case supplies the required deletion set, completing the proof of the upper bound.

For $n\ge15$, Lan et~al.~\cite[Theorem~6]{LanShiSong2019} give planar graphs of maximum degree at most five with $\lfloor5n/2\rfloor$ edges. They are $W_{2,5}$-free, since this tree has a vertex of degree six. These graphs attain the upper bound and complete the proof of Theorem~\ref{thm:w25}.
\end{proof}

\subsection{The order of a \texorpdfstring{$W_{2,5}$}{W(2,5)}-free triangulation}

\begin{corollary}\label{prop:w25-triangulation}
A $W_{2,5}$-free plane triangulation has at most twelve vertices. The icosahedral graph attains this bound.
\end{corollary}

\begin{proof}
If $G$ is such a triangulation on $n$ vertices, then, by Theorem~\ref{thm:w25},
\[
3n-6=e(G)\le\frac52n,
\]
and therefore $n\le12$. The icosahedral graph in Figure~\ref{fig:sharp-triangulations}(b) is a twelve-vertex triangulation of maximum degree five, so it is $W_{2,5}$-free and attains the bound.
\end{proof}

\section{Dominating vertices and triangulations for \texorpdfstring{$W_{3,4}$}{W(3,4)}}
\label{sec:w34}

We first prove the characterization and the exact extremal result in
Theorem~\ref{thm:w34}.

\begin{proof}[Proof of Theorem~\ref{thm:w34}]
Put $F=G-v$ and $s=n-1$. The graph $F$ is outerplanar: deleting the
dominating vertex $v$ from a plane embedding exposes every remaining
vertex to the face formerly containing $v$.

Suppose that $d_F(u)\ge4$ for some vertex $u$. Choose four distinct
neighbors $w,a_1,a_2,a_3$ of $u$ in $F$. Use $uwv$ as a spine and
$a_1,a_2,a_3$ as the three leaves at $u$. The six vertices already used
leave at least $n-6\ge4$ other vertices, all adjacent to $v$, for the
four leaves at $v$. Thus $G$ contains $W_{3,4}$.

Conversely, suppose that $\Delta(F)\le3$. Every vertex of $G$ other
than $v$ has degree at most four. In a copy of $W_{3,4}$ the endpoint
supporting four leaves would therefore have to be $v$. The other
endpoint $u$ would need three leaves and a middle vertex, all distinct
from $v$. Since $u$ is also adjacent to $v$ in $G$, this would require
$d_G(u)\ge5$, a contradiction. This proves the characterization.

Every outerplanar graph on $s\ge3$ vertices has at least two vertices
of degree at most two: complete an outerplane embedding to a maximal
outerplanar graph on the same vertex set and use its two ears. Since
$\Delta(F)\le3$, we have
\[
2e(F)\le3s-2,
\qquad
e(G)=s+e(F)\le\left\lfloor\frac{5s-2}{2}\right\rfloor
=\left\lfloor\frac{5n-7}{2}\right\rfloor.
\]

For completeness, we construct an outerplanar graph $F_s$ with
$\Delta(F_s)\le3$ and $e(F_s)=\lfloor(3s-2)/2\rfloor$ for every
$s\ge8$. If $s=2k$, start with the ladder having vertices
\[
p_1,\ldots,p_{k-1},q_1,\ldots,q_{k-1}
\]
and edges $p_iq_i$ for $1\le i\le k-1$, together with
$p_ip_{i+1}$ and $q_iq_{i+1}$ for $1\le i<k-1$. Add two vertices
$a,b$, adjacent respectively to both ends of the left rung
$p_1q_1$ and of the right rung $p_{k-1}q_{k-1}$. They can be placed
outside the two ends of the ladder, giving an outerplane embedding
(Figure~\ref{fig:ear-ladder}). The original ladder has $2k-2$ vertices
and $3k-5$ edges. Consequently $F_{2k}$ has $2k$ vertices and $3k-1$
edges. Its two added vertices have degree two, and every other vertex
has degree three.

If $s=2k+1$, subdivide the outer boundary edge $ap_1$ of $F_{2k}$
once. This preserves outerplanarity and the degrees of its existing
vertices, and the new vertex has degree two. It yields
$e(F_{2k+1})=3k=\lfloor(3s-2)/2\rfloor$.

Finally, add a new vertex $v$ adjacent to every vertex of $F_s$.
This is planar: draw these new edges in the complementary disk on
the other side of the outer boundary cycle of $F_s$. The
characterization already proved shows that the resulting graph is
$W_{3,4}$-free. Its edge count attains the asserted bound for every
$n=s+1\ge10$.
\end{proof}

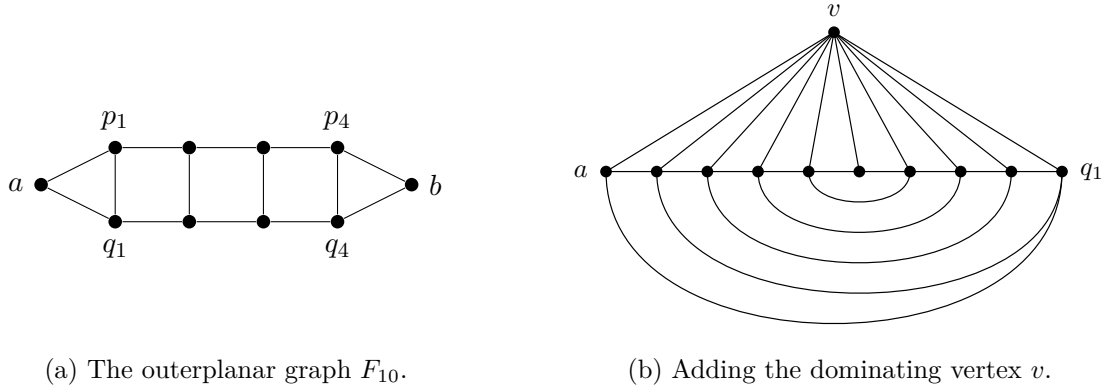
\begin{figure}[htbp]
\centering
\begin{minipage}[b]{.44\textwidth}
\centering
\begin{minipage}[c][4.3cm][c]{\linewidth}
\centering
\begin{tikzpicture}[scale=.98,vertex/.style={circle,fill=black,inner sep=1.8pt}]
\foreach \i in {1,2,3,4}{
  \node[vertex] (p\i) at (\i-1,1) {};
  \node[vertex] (q\i) at (\i-1,0) {};
  \draw (p\i)--(q\i);
}
\foreach \i/\j in {1/2,2/3,3/4}{
  \draw (p\i)--(p\j);
  \draw (q\i)--(q\j);
}
\node[vertex,label=left:$a$] (a) at (-1,.5) {};
\node[vertex,label=right:$b$] (b) at (4,.5) {};
\draw (a)--(p1) (a)--(q1) (b)--(p4) (b)--(q4);
\node[above=3pt] at (p1) {$p_1$};
\node[below=3pt] at (q1) {$q_1$};
\node[above=3pt] at (p4) {$p_4$};
\node[below=3pt] at (q4) {$q_4$};
\end{tikzpicture}
\end{minipage}
\par\smallskip
{\small (a) The outerplanar graph $F_{10}$.}
\end{minipage}\hfill
\begin{minipage}[b]{.54\textwidth}
\centering
\begin{minipage}[c][4.3cm][c]{\linewidth}
\centering
\begin{tikzpicture}[x=.67cm,y=.67cm,
  every node/.style={font=\small},
  vertex/.style={circle,fill=black,inner sep=1.55pt}]
  \coordinate (a) at (0,0);
  \coordinate (p1) at (1,0);
  \coordinate (p2) at (2,0);
  \coordinate (p3) at (3,0);
  \coordinate (p4) at (4,0);
  \coordinate (b) at (5,0);
  \coordinate (q4) at (6,0);
  \coordinate (q3) at (7,0);
  \coordinate (q2) at (8,0);
  \coordinate (q1) at (9,0);
  \coordinate (v) at (4.5,2.75);
  \draw[line width=.45pt]
    (a)--(p1)--(p2)--(p3)--(p4)--(b)--(q4)--(q3)--(q2)--(q1);
  \draw[line width=.45pt]
    (a) .. controls (0,-4) and (9,-4) .. (q1);
  \foreach \i in {1,...,4} {
    \pgfmathsetmacro{\depth}{.8*(5-\i)}
    \draw[line width=.45pt]
      (p\i) .. controls (\i,-\depth) and ({10-\i},-\depth) .. (q\i);
  }
  \foreach \x in {a,p1,p2,p3,p4,b,q4,q3,q2,q1} {
    \draw[line width=.45pt] (v)--(\x);
    \node[vertex] at (\x) {};
  }
  \node[vertex,label=above:$v$] at (v) {};
  \node[left=3pt] at (a) {$a$};
  \node[right=3pt] at (q1) {$q_1$};

\end{tikzpicture}
\end{minipage}
\par\smallskip
{\small (b) Adding the dominating vertex $v$.}
\end{minipage}
\caption{The construction for Theorem~\ref{thm:w34}, shown for $k=5$. The base graph $F_{10}$ in (a) has ten vertices and fourteen edges. In the plane drawing (b), adding $v$ adjacent to all ten base vertices produces a $W_{3,4}$-free extremal graph with $n=11$ and $e=24=\lfloor(5n-7)/2\rfloor$. The spokes lie above the horizontal path and the remaining five base edges below it. Subdividing the outer boundary edge $ap_1$ in (a) before adding $v$ illustrates the construction for an odd number of base vertices.}
\label{fig:ear-ladder}
\end{figure}

We turn to unrestricted triangulations. Unlike the dominating-vertex
result, the next proposition is only an upper bound on the order.

\begin{proposition}\label{prop:w34tri}
Every $W_{3,4}$-free plane triangulation has at most eleven vertices.
\end{proposition}

\begin{proof}
Suppose that $G$ is a $W_{3,4}$-free plane triangulation on $n\ge12$
vertices. We repeatedly use Lemma~\ref{lem:spine}, the minimum degree
bound $\delta(G)\ge3$, and Lemma~\ref{lem:adjacent-three}.

\medskip\noindent
\textbf{Case 1: $\Delta(G)\ge8$.}
Choose $x$ with $d(x)\ge8$. Every vertex $z\in N_2(x)$ has degree
three. Indeed, if $d(z)\ge4$, choose a common neighbor $y$ of $x,z$
and three neighbors of $z$ other than $y$. These are valid leaves at
$z$, and deleting them and $y$ leaves at least $d(x)-4\ge4$
neighbors of $x$ available as leaves at $x$.

For such a degree-three vertex $z$, its three neighbors form a
triangle. If it had a neighbor $w\notin N(x)$, a common neighbor
$y$ of $x,z$ would also be adjacent to $w$. Thus $w\in N_2(x)$,
so $d(w)=3$, contradicting Lemma~\ref{lem:adjacent-three}. Therefore
$N(z)\subseteq N(x)$ for every $z\in N_2(x)$, and there are no
vertices at distance three or more from $x$.

Suppose that a vertex $z$ exists outside $N[x]$, and write
$N(z)=\{a,b,c\}$. The vertices $a,b,c$ form a triangle and all
belong to $N(x)$, so each has degree at least four. If, for example,
$d(a)\ge5$, then
\[
\cdeg(x,a)\le d(a)-2,
\]
because both $x$ and $z$ belong to $N(a)\setminus N(x)$.
The adjacent endpoints $a,x$ have a common neighbor, and their
available leaf-neighborhood union on the corresponding spine has
size at least
\[
d(x)+d(a)-3-\cdeg(x,a)\ge d(x)-1\ge7.
\]
Their degrees satisfy the requirements of Lemma~\ref{lem:spine},
giving $W_{3,4}$. Hence $a,b,c$ all have degree exactly four.
The five vertices $x,z,a,b,c$ induce the triangular bipyramid,
and the degrees of $a,b,c$ are exhausted by these five vertices.
Each of the six faces of this bipyramid contains an edge with both
ends among $a,b,c$. No such face can contain a further vertex of
the triangulation: the triangular face incident with that edge
inside it would introduce an additional neighbor of an exhausted
vertex. Thus $G$ would have only five vertices, a contradiction.

It follows that $x$ is dominating. Theorem~\ref{thm:w34} implies
that every other vertex has degree at most four. By Euler's formula,
we then have
\[
6n-12=\sum_{u\in V(G)}d(u)\le5(n-1),
\]
forcing $n\le7$, again a contradiction.

\medskip\noindent
\textbf{Case 2: $\Delta(G)=7$.}
Choose $x$ with $d(x)=7$. If $z\in N_2(x)$ has degree at least
four, then $N(z)\subseteq N(x)$. Otherwise, choose a common
neighbor $y$ of $x,z$, one neighbor of $z$ outside $N(x)$, and
two other neighbors of $z$ distinct from $y$. Using the last three
vertices as leaves at $z$ excludes at most three neighbors of $x$,
including $y$, and leaves four available leaves at $x$.

The same containment holds when $d(z)=3$. If $z$ had a neighbor
$w\notin N(x)$, the triangular neighborhood of $z$ would imply
$w\in N_2(x)$. If $d(w)=3$, this contradicts
Lemma~\ref{lem:adjacent-three}. If $d(w)\ge4$, the containment
just proved for $w$ is contradicted by its neighbor $z\notin N(x)$.
Consequently there are no vertices beyond distance two from $x$,
and $S=V(G)\setminus N[x]$ is independent.

Every $u\in N(x)$ has at most one neighbor in $S$. For $d(u)\ge5$,
apply Lemma~\ref{lem:spine} to the adjacent endpoints $u,x$ and
any common neighbor. The degree requirements hold, so freeness
forces
\[
d(u)-\cdeg(x,u)\le2.
\]
As $x\in N(u)\setminus N(x)$, this permits at most one neighbor
of $u$ in $S$. If $d(u)=4$, the two triangular faces on $xu$
provide two distinct neighbors of $u$ in $N(x)$, again leaving at
most one in $S$; if $d(u)=3$, they leave none. It follows that
$3|S|\le e(S,N(x))\le7$, and hence $n=8+|S|\le10$, a contradiction.

\medskip\noindent
\textbf{Case 3: $\Delta(G)=6$.}
Choose $x$ with $d(x)=6$. We claim that
\begin{equation}\label{eq:w34-distance-two}
|N(z)\setminus N(x)|\le1
\qquad\text{for every }z\in N_2(x).
\end{equation}
First let $d(z)\ge4$. If $z$ had two neighbors outside $N(x)$,
choose them and one further neighbor other than a fixed common
neighbor $y$ of $x,z$ as three leaves at $z$. Together with $y$,
these leaves exclude at most two neighbors of $x$, leaving four
leaves at $x$. This proves~\eqref{eq:w34-distance-two} in this
case. If $d(z)=3$ and its neighbors included two vertices outside
$N(x)$, the triangular neighborhood of $z$ would make both
vertices belong to $N_2(x)$. Neither can have degree three by
Lemma~\ref{lem:adjacent-three}. Each would therefore contradict
the inequality already proved, since its neighbors outside
$N(x)$ include $z$ and the other of these two vertices.

A vertex $z\in N_2(x)$ cannot be adjacent to a vertex $w$ at
distance three from $x$. Indeed, a triangular face on $zw$ has a
third vertex $t\in N(z)$, and~\eqref{eq:w34-distance-two} forces
$t\in N(x)$. But then $w$ is at distance two from $x$. Hence all
vertices of $S=V(G)\setminus N[x]$ have distance two from $x$,
and~\eqref{eq:w34-distance-two} implies $\Delta(G[S])\le1$.
Write $r=|S|=n-7$.

Using Euler's formula and the maximum degree on $N[x]$, we obtain
\begin{equation}\label{eq:w34-outside-lower}
\sum_{z\in S}d(z)\ge6n-12-42=6r-12.
\end{equation}
Retain only the edges between $N(x)$ and $\{x\}\cup S$. The
resulting bipartite planar graph has $r+7$ vertices, so
\[
6+\sum_{z\in S}d(z)-2e(S)\le2(r+7)-4.
\]
Since $e(S)\le\lfloor r/2\rfloor$, this yields
\begin{equation}\label{eq:w34-outside-upper}
\sum_{z\in S}d(z)\le2r+4+2\lfloor r/2\rfloor\le3r+4.
\end{equation}
The bounds~\eqref{eq:w34-outside-lower} and
\eqref{eq:w34-outside-upper} are incompatible for $r\ge6$.
If $n=12$, then $r=5$ and the lower bound and the first upper
bound both equal eighteen. Equality in the lower bound forces
every vertex of $N[x]$ to have degree six. The outerplanar graph
$G[N(x)]$ has a vertex $u$ of degree at most two within that
graph. Thus $\cdeg(x,u)\le2$, while $d(x)=d(u)=6$. The endpoints
$x,u$ are adjacent and have a common neighbor; on a spine
through it their available leaf-neighborhood union has size at
least $6+6-3-2=7$. Lemma~\ref{lem:spine} gives $W_{3,4}$,
again a contradiction.

\medskip\noindent
\textbf{Case 4: $\Delta(G)\le5$.}
By Euler's formula, $6n-12\le5n$, so the assumption $n\ge12$
forces $n=12$ and $G$ to be 5-regular. We show directly that
this graph contains $W_{3,4}$.

First, $G$ has no separating triangle. To see this, let one side
of such a triangle contain $i$ vertices, let $m$ count the edges
with both ends among those vertices, and let $b$ count the edges
from them to the boundary triangle. Counting edges in the triangulated disk
and summing the degrees of the interior vertices, respectively, we obtain
\[
m+b=3i,\qquad 2m+b=5i.
\]
Hence $m=2i$. This is impossible for $i=1,2$; for $i\ge3$,
we have $2i=m\le3i-6$ by planarity, so $i\ge6$. Applying the same
argument on the other side would force at least $6+6+3=15$
vertices, contrary to $n=12$.

Fix a vertex $x$, and list its neighbors in cyclic order as
$u_1,\ldots,u_5$, with indices read modulo five. They induce a
chordless cycle: a chord would form a separating triangle with
$x$. Every $u_i$ has precisely two neighbors outside $N[x]$,
its three other neighbors being $x,u_{i-1},u_{i+1}$. Let $z_i$
be the third vertex of the triangular face on $u_iu_{i+1}$
opposite $x$. As the neighbor cycle is chordless, $z_i$ lies
outside $N[x]$.

Consecutive labels $z_{i-1},z_i$ are distinct. Otherwise, the
two faces on the corresponding rim edges together with the two
faces containing $x$ would close the rotation at $u_i$ with only
four neighbors, contradicting $d(u_i)=5$. Therefore
$z_{i-1},z_i$ are precisely the two neighbors of $u_i$ outside
$N[x]$.

A label can occur at most twice in the cyclic list
$z_1,\ldots,z_5$, since consecutive occurrences are excluded.
Because the list has odd length, some label $z$ occurs exactly
once. It follows that $x,z$ have exactly two common neighbors,
namely the ends of the uniquely labeled rim edge: any additional
neighbor $u_j$ would force $z=z_{j-1}$ or $z=z_j$ and give another
occurrence. The vertices $x,z$ are nonadjacent and both have
degree five. On a spine through either common neighbor, their
available leaf sets each have size four and their union has size
$5+5-1-2=7$. Lemma~\ref{lem:spine} yields $W_{3,4}$, the final
contradiction.
\end{proof}

\begin{remark}
Every triangulation on at most nine vertices is automatically
$W_{3,4}$-free, since the forbidden tree has ten vertices.
Proposition~\ref{prop:w34tri} does not settle whether there exists
a $W_{3,4}$-free triangulation on ten or eleven vertices.
\end{remark}

\section{Concluding questions}

Theorem~\ref{thm:w24} determines the planar Tur\'an density of $W_{2,4}$ and its exact extremal number when the order is divisible by eight. Determining $\ex_{\PP}(n,W_{2,4})$ in the remaining residue classes is still a finite-order problem. A related question is to characterize all graphs attaining the bound at multiples of eight.

Theorem~\ref{thm:w25} determines $\ex_{\PP}(n,W_{2,5})$ for every $n\ge15$. It remains natural to characterize the extremal graphs. The known constructions have maximum degree five; understanding whether extremal graphs with vertices of larger degree occur at arbitrarily large orders would provide a more structural description of the equality case.

For $W_{3,4}$, the extremal function is determined here only for graphs with a dominating vertex. Its unrestricted asymptotic coefficient, and the existence of $W_{3,4}$-free triangulations on ten or eleven vertices, are not settled by these arguments. Extending the degree-deficit argument to quasi-double stars with more leaves is a further possible direction.

\end{document}